\documentclass[a4paper,12pt,notitlepage]{amsart}
\usepackage{amsfonts, amsmath, amscd, amsthm, graphicx, amssymb}
\usepackage{tikz}
\usepackage{multicol}
\usetikzlibrary{arrows,shapes,decorations.markings,calc,positioning}
\usepackage{enumitem}

\usepackage{stackengine}
\usepackage{indentfirst}
\usepackage{mathtools}
\usepackage{tikz-cd} 
\usepackage{url} 
\usepackage{mathrsfs}
\usepackage{comment}
\theoremstyle{definition}
\newtheorem{theorem}{Theorem}[section]
\newtheorem*{theorem*}{Theorem}

\newtheorem{proposition}[theorem]{Proposition}
\newtheorem{lemma}[theorem]{Lemma}
\newtheorem{example}[theorem]{Example}
\newtheorem{corollary}[theorem]{Corollary}
\newtheorem{remark}[theorem]{Remark}

\theoremstyle{definition}
\newtheorem*{acknowledgement*}{Acknowledgements}
\newtheorem*{problem*}{Problem}
\theoremstyle{definition}
\newtheorem{definition}[theorem]{Definition}

\makeatletter
\newcommand*{\rom}[1]{\expandafter\@slowromancap\romannumeral #1@}
\makeatother

\newcommand{\id}{\text{id}}
\newcommand{\M}{\mathcal{M}}
\newcommand{\Tsim}{\stackrel T\sim }
\newcommand{\R}{\mathbb{R}}

\newcommand{\es}{\tilde{\emptyset}}
\newcommand{\MMK}{\widetilde{\mathcal{M}} (K)}
\newcommand\stackrqarrow[2]{%
    \mathrel{\stackunder[2pt]{\stackon[4pt]{$\rightsquigarrow$}{$\scriptscriptstyle#1$}}{%
            $\scriptscriptstyle#2$}}}
\title{Connecting Discrete Morse Functions via Birth-Death Transitions}
\author{Chong Zheng}
\address{Faculty of  Science and Engineering, Waseda University, 
Ohkubo, Shinkuju-ku, Tokyo, 169-8555 Japan}
\email{c{\_}zheng@aoni.waseda.jp}
\keywords{Discrete Morse theory, birth-death transitions, simplicial complexes, complex of discrete Morse functions.}

\begin{document}

\begin{abstract}
 We study transformations 
 between discrete Morse functions on a finite simplicial complex via birth-death transitions—elementary chain maps between discrete Morse 
 complexes that  either create or cancel pairs of critical simplices.
We prove that any two discrete Morse
functions
$f_1,f_2:K \to \R$ on a finite simplicial complex 
$K$ 
are linked by a finite sequence of such transitions.
As  applications, we present alternative proofs 
of several of Forman’s fundamental results in discrete Morse theory
and study the topology  of  the space of discrete Morse functions.
\end{abstract}
\maketitle

\section{Introduction}
In classical Morse theory,
a well-known result states that  
two Morse functions $\phi_1, \phi_2:M \to \R$
on a compact manifold $M$
can be continuously deformed into one another via a one-parameter 
family \cite{cerf,cerf2}.
 This fundamental result 
 plays a crucial role in understanding 
 the structure of differentiable manifolds
  by providing a mechanism to relate different Morse functions on the same manifold. 

In discrete Morse theory, 
a  parallel result is naturally expected.
Discrete Morse theory, 
introduced by Forman \cite{Forman_guide, Morse Theory for Cell Complexes}, 
adapts  the philosophy of 
smooth cases  to simplicial complexes and cell complexes, 
replacing smooth functions with combinatorial analogs.
For a given simplicial complex 
$K$, 
a discrete Morse function assigns values to its simplices in a manner compatible
 with topological gradient flow.

Let $K$ be a simplicial complex
and $f_1, f_2:K \to \R$ be two discrete Morse functions. 
In \cite{birth and death} and our previous works \cite{no.1,no.2},  
we introduced   the concept of  \textit{connectedness} between both
critical simplices and discrete Morse functions 
enabling 
an algebraic characterization of the relationship between their associated \textit{gradient vector fields}. 
In particular,
the \textit{connectedness homomorphisms}
between discrete Morse complexes
$C_{\ast}^{f_1}, C_{\ast}^{f_2}$,
corresponding  to $f_1$ and $f_2$ respectively,
provide a powerful tool for analyzing the relationships between 
$f_1$ and $f_2$.

Another central theme in discrete Morse theory 
 is the study of \textit{the complex of discrete Morse functions},
denoted $\M(K)$, as
introduced by Chari and Joswig \cite{chari}.
The authors constructed  
$\M(K)$ as a simplicial complex to encode 
the space of 
discrete Morse functions on $K$,
establishing a bijection between $\M(K)$
and the set of gradient vector fields on $K$.
Following this work,
 investigating the  homotopy type  of $\M(K)$ 
has been a major task in discrete Morse theory \cite{determined,homotopy_type,homotopy_type 2,higher connectivity of the morse complex}.

This paper 
continues the study and clarifies some ambiguities  in \cite{no.2} by
 following the idea of  \textit{birth-death transitions}
to connect discrete Morse functions on $\M(K)$.
Precisely, 
we show in Theorem \ref{main_theorem} that 
any two  discrete Morse functions 
$f_1$ and $f_2$  can be transformed into 
one another
with a finite sequence of 
birth-death transitions.
Moreover, 
we apply this theorem to give an alternative, self-contained proof of some
Forman's fundamental results over $\mathbb{Q}$ in 
Theorem \ref{thm:homology-isomorphism}
and 
Theorem \ref{thm:morse_equality}.

This paper is organized as follows.
In Section 2, 
we provide a
 concise overview of the 
  preliminaries, 
 including 
 the fundamentals of
  discrete Morse theory,	
  as well as the concepts
   of connectedness homomorphisms.
In Section 3,
we introduce the birth-death transitions introduced in \cite{no.2}
between two discrete Morse complexes.
In Section 4,
we prove Theorem \ref{main_theorem}.
Section 5 applies this framework to re-establish key theorems of Forman 
through the compositions of birth-death transitions.
Also,
we provide some remarks about the topology of the space of 
discrete Morse functions on  $K$.
\section{Preliminaries}
This section provides the foundational concepts necessary for the remainder of the paper.

\subsection{Discrete Morse Functions and Gradient Vector Fields}
For a comprehensive
understanding 
of  \textit{discrete Morse theory } and \textit{discrete Morse homology},
readers are referred  
 to  \cite{Forman_guide, Morse Theory for Cell Complexes,Discrete Morse Theory Kozlov}.

Let $K$ be a simplicial complex,
$K_p$ be the set of $p$-dimensional simplices.
We write $\sigma^{(p)}$ to represent a 
$p$-dimensional simplex,
omitting the superscript
$p$
when the dimension is clear from context.
We write $\tau \prec \sigma$  and  $\tau  \succ  \sigma$ if 
\textit{$\tau $ is a face of $\sigma$}, and \textit{$\sigma$ is a face of $\tau $}.

A \textit{discrete Morse function}
$f:K \to \mathbb{R}$
satisfies for every simplex
$\sigma \in K_p$,
\begin{enumerate}
\item $\displaystyle \# \{\tau^{(p+1)} \succ \sigma \, | \, f(\tau) \leq f(\sigma)\} = 0 \text{ or } 1;$
\item 
$\displaystyle  \# \{v^{(p-1)} \prec \sigma \, | \, f(\sigma) \leq f(v)\} =0 \text{ or } 1,$
\end{enumerate}
where $\#$ denotes the cardinality of a set.

A simplex
$\sigma^{(p)}$ is called a
\textit{critical}  simplex with respect to $f$
if:
\begin{enumerate}
\item 
$\displaystyle\# \{\tau^{(p+1)} \succ \sigma \, | \, f(\tau) \leq f(\sigma)\} =0;$
\item 
$\displaystyle\# \{v^{(p-1)} \prec \sigma \, | \, f(\sigma) \leq f(v)\} =0.$
\end{enumerate}
If $\sigma$ is  critical,
then we call $f(\sigma)$ a \textit{critical value} of $f$.
Conversely, if a simplex $\sigma$ is not critical, 
then
we say that $\sigma$ is a \textit{non-critical} simplex,
and $f(\sigma)$ is a \textit{non-critical value}.
The set of $f$-critical $q$-simplices is denoted
$Cr_q(f):= \{ \sigma \in K_q \, | \, \sigma^{(q)} \text{ is } f \text{-critical}  
  \}$  .

The  \textit{gradient vector field} $V$ corresponding to $f$ is 
the collection of pairs 
$(\alpha^{(p)}, \beta^{(p+1)})$
satisfying 
$\alpha \prec \beta$ and
$f(\alpha) \geq f(\beta)$.
We call such a pair $(\alpha, \beta)$
a \textit{gradient pair}.

A  \textit{gradient $V$-path} is a sequence 
of simplices
$$\alpha_0^{(p)}, \beta_0^{(p+1)}, \alpha_1^{(p)}, \beta_1^{(p+1)},\cdots, \beta_r^{(p+1)}, \alpha_{r+1}^{(p)} $$
such that for
each $i= 0, \cdots,r $:
\begin{itemize}
\item $(\alpha^{(p)}, \beta^{(p+1)})\in V$,
\item $\beta_i^{(p+1)} \succ \alpha_{i+1}^{(p)}$ and $\alpha_{i+1}^{(p)} \neq \alpha_i^{(p)}.$
\end{itemize}

We denote a \textit{gradient  $V$-path} $\gamma$ 
from 
$\alpha_0^{(p)} $ to 
$\alpha_{r+1}^{(p)}$,
consisting of $p$- and $(p+1)$-dimensional 
simplices
by 
$$\gamma: \alpha_0^{(p)}  \stackrqarrow{V(p,p+1)}{}  \alpha_{r+1}^{(p)}. $$

Alternatively, a $V$-path may begin at a 
$(p+1)$-dimensional simplex $\beta_0$ and end at $\beta_r$,
denoted:
$$\gamma: \beta_0^{(p+1)} \stackrqarrow{V(p,p+1)}{}  \beta_r^{(p+1)}. $$

\subsection{Discrete Morse Complexes}

As a discrete analog of the 
Morse-Smale complex in smooth Morse theory (see \cite{Milnor, Guest}), 
the discrete Morse complex is defined  by Forman \cite{Forman_guide}.
Its chain groups 
are generated by critical simplices, with 
boundaries determined by counting 
gradient paths weighted by incidence numbers.

Assume each simplex of $K$ is oriented.
For $\alpha^{(q-1)}\prec \sigma^{(q)}$,
we use $[ \sigma : \alpha]=\pm 1$ to represent \textit{the incidence number} of 
 $\sigma$ and $ \alpha$,
where the sign is chosen according to the orientation.

Let $f:K\to \mathbb{R}$ be a discrete Morse function.
The \textit{$q$-th chain group of discrete Morse complex} 
$C_q^f= C_q^f(K)$
is the free Abelian group generated by
 $\sigma_i \in Cr_q(f)$,  
The
\textit{$q$-th boundary homomorphism}
$\partial_{q}^{f}: C_q^f \to C_{q-1}^f$  
is defined as the linear extension 
of
$$\partial_{q}^{f}(\sigma_i):= \sum_{ \alpha \in Cr_{q-1}(f)} n^{f}(\sigma_i, \alpha) \alpha.$$
Here, $n^{f}(\sigma_i, \alpha)$ is 
given by
$$n^{f}(\sigma_i, \alpha):= \sum_{\gamma } m(\gamma),$$
where $\gamma$ 
runs over all gradient paths $\sigma_i \stackrqarrow{V(q-1,q)}{}\alpha$,
and
$m(\gamma)$ is given by the product of 
each incidence number.
We call $n^{f}(\sigma_i, \alpha)$
the \textit{connectedness coefficient between  
$\sigma_i$ and  $\alpha$}.
In particular, $\partial_{0}^{f}$ is defined as
 the zero homomorphism.

It is known that 
$\partial^f \circ \partial^f=0$ \cite[Theorem 7.1]{Forman_guide}.  
Hence, $\{C_{\ast}^f, \partial_{\ast}^f \}$ 
forms the \textit{discrete Morse chain complex},
and the \textit{discrete Morse homology}
is defined as 
$$H_{\ast}^f:= \frac{\text{ker} \partial^f_{\ast} }{\text{im} \partial^f_{\ast}}.$$
Moreover, discrete Morse homology is 
isomorphic to the ordinary homology \cite[Theorem 7.1]{Forman_guide}.

\subsection{Connectedness and  Connectedness Homomorphisms}

Let $f_1, f_2: K \to \mathbb{R} $ be distinct discrete
Morse functions on $K$,
with respective gradient vector fields $V_1$ and $V_2$.
The notion of  \textit{connectedness} between critical simplices $\tilde{\sigma}_1\in Cr_{q}(f_1)$ and 
$\tilde{\sigma}_2\in Cr_q(f_2)$ 
is first introduced by King, Mramor Kosta, and Knudson \cite{birth and death}.
They applied the concept to investigate the structure of the gradient vector fields 
and provided some interesting applications. 

Precisely,
for critical simplices $\tilde{\sigma}_1\in Cr_{q}(f_1)$ and 
$\tilde{\sigma}_2\in Cr_q(f_2)$, 
we  say that 
$\tilde{\sigma}_1$ 
is \textit{connected} to $\tilde{\sigma}_2$ if
\begin{itemize}
\item for $q\neq 0$, 
 there is an $f_1$-gradient path  
$\tilde{\sigma}_1^{(q)} \stackrqarrow{V_1(q-1,q)}{} \tilde{\sigma}_2^{(q)}; $

\item for $q=0$, 
there is an $f_2$-gradient path  
$\tilde{\sigma}_1^{(0)} \stackrqarrow{V_2(0,1)}{} \tilde{\sigma}_2^{(0)}.$
\end{itemize}

Note that the corresponding connectedness coefficient $n^{f_1}(\tilde{\sigma}_1,   \tilde{\sigma}_2)$ captures the number (weighted by incidence) of such 
gradient paths.

Furthermore,
we say that
$\tilde{\sigma}_1$ 
is \textit{strongly connected} to $\tilde{\sigma}_2$
if $\tilde{\sigma}_1$ 
is connected to $\tilde{\sigma}_2$ and 
$\tilde{\sigma}_2$ 
is connected to $\tilde{\sigma}_1$,
denoted    
$\tilde{\sigma}_1 \leftrightarrow \tilde{\sigma}_2.$

For each $q\geq 1$,
the \textit{$q$-dimensional connectedness homomorphism}
from $C_{q}^{f_1}$ to
$C_{q}^{f_2}$
$$h_q: C_{q}^{f_1} \to C_{q}^{f_2}$$
is defined as the linear extension of 
$$h_q(\tilde{\sigma}_1):= \sum_{\sigma_j\in Cr_q(f_2)} n^{f_1}(\tilde{\sigma}_1, \sigma_j) \sigma_j,$$
where $\tilde{\sigma}_1 \in Cr_q(f_1)$.

For $q=0$,
inheriting the $0$-dimensional connectedness between simplices $\tilde{\sigma}_1^{(0)}$ and 
$\tilde{\sigma}_2^{(0)}$, 
we let 
$$h_0: C_{0}^{f_1} \to C_{0}^{f_2}$$
be the linear extension of 
$$h_0(\tilde{\sigma}_1):= \sum_{\sigma_j\in Cr_0(f_2)} n^{f_2}(\tilde{\sigma}_1, \sigma_j) \sigma_j,$$
where $\tilde{\sigma}_1 \in Cr_0(f_1)$.

On the other hand, 	
we can also define the reverse \textit{connectedness homomorphism}
\textit{from} \textit{$C_q^{f_2}$ to
$C_q^{f_1}$}
with a parallel way 
as 
$$g_q: C_{q}^{f_2} \to C_{q}^{f_1}.$$

Note, 
however, 
that these maps are 
generally not chain maps, 
since connectedness is not transitive in general.

\section{Birth-death transitions between discrete Morse complexes}
In this section,
we review and provide more details on
\textit{birth-death transitions} introduced in \cite{no.2}.
The notion of birth-death transitions  is  a discrete analogue inspired
 by classical Morse theory’s creation and cancellation of critical points. 
These transitions act as chain maps between discrete Morse complexes
$C_{\ast}^{f_1}$
and 
$C_{\ast}^{f_2}$
adding or removing a pair of critical simplices in consecutive dimensions,
while preserving the underlying homology \cite{no.2}.
Building upon the 
connectedness homomorphisms defined previously, we now formalize the birth and death transitions and investigate their structural properties.

Fix a dimension $i$.
Suppose two discrete Morse functions 
$f_1, f_2: K\to \R$ on a simplicial complex $K$ such that:
$$\#Cr_i(f_1)= \# Cr_i(f_2) -1;\quad \#Cr_{i-1}(f_1)=\# Cr_{i-1}(f_2) -1,$$
and for all  $q \neq i, i-1$, 
$$Cr_q(f_1)= Cr_q(f_2).$$
We denote the two additional $f_2$-critical simplices
simplices as 
$\tilde{\sigma}\in Cr_i(f_2)$
and
$\tilde{\alpha}\in Cr_{i-1}(f_2)$,
with the coefficient
$n^{f_2}(\tilde{\sigma},\tilde{\alpha})=k\in \mathbb{Z}.$

 \begin{definition}\label{birth-death transition}
Under the conditions specified above,  
we define
\begin{itemize}
    \item The connectedness homomorphism $h : C_q^{f_1} \to C_q^{f_2}$ as a \textit{birth transition};
    \item The connectedness homomorphism $g : C_q^{f_2} \to C_q^{f_1}$ as a \textit{death transition}.
\end{itemize}
 
\end{definition}

In  \cite{no.2},
 the following properties were assumed 
for simplicity.
Here we 
state them explicitly as lemmas to clarify their behaviours.
Note that these lemmas simply follow from 
the assumptions in the definition of 
birth-death transitions.

\begin{lemma}
Let 
$h:C_q^{f_1} \to C_q^{f_2}$ and $g:C_q^{f_2} \to C_q^{f_1}$ be a birth transition and 
a death transition, 
respectively.

Then, 
\begin{itemize}

\item For any  $\delta_1\in Cr_q(f_1)$, 
$$h_q(\delta_1)= 
\begin{cases}
\delta_1 & \text{ if }q\neq i; \\
\delta_1+n^{f_1}(\delta_1, \tilde{\sigma}) \tilde{\sigma} &\text{ if } q=i.

\end{cases}$$

\item For any  $\delta_2\in Cr_q(f_2)$, 
$$g_q(\delta_2)= 
\begin{cases}
\delta_2 & \text{ if }\delta_2 \neq \tilde{\sigma}, \tilde{\alpha};  	\\
0 &\text{ if } \delta_2=\tilde{\sigma};		\\
\displaystyle  \sum_{\alpha\in Cr_{i-1}(f_1)} n^{f_2} (\tilde{\alpha}, \alpha) \alpha & \text{ if }  \delta_2=\tilde{\alpha}.

\end{cases}$$

\end{itemize}

\end{lemma}

\begin{lemma}
Let 
$h:C_q^{f_1} \to C_q^{f_2}$ and $g:C_q^{f_2} \to C_q^{f_1}$ be a birth transition and 
a death transition, 
respectively.

Then, with respect to the boundary homomorphisms 
$\partial_{\ast}^{f_1}$ and $\partial_{\ast}^{f_2}$,
the following equations (1) to (5) 
hold.

Namely, when  $ q\neq i, i+1   \text{ and } 
\delta\neq \tilde{\alpha},$
\begin{equation}\label{1}
\partial^{f_1}_q (\delta)= \partial^{f_2}_q (\delta).
\end{equation}

For any  $\tau \in Cr_{i+1}(f_2)$,
\begin{equation} \label{2}
\partial_{i+1}^{f_2}(\tau )=  \partial_{i+1}^{f_1}(\tau ) + \sum_{\sigma \in Cr_{i}(f_1)} n^{f_1} (\tau, \sigma)\cdot n^{f_1}(\sigma, \tilde{\sigma})  \tilde{\sigma}.
\end{equation}
 For any $\sigma \in Cr_{i}(f_2)$ and $\sigma \neq \tilde{\sigma}$,
 \begin{equation} \label{3}
\partial_{i}^{f_2}(\sigma )=\partial_i^{f_1}(\sigma) + n^{f_1}(\sigma, \tilde{\sigma})   (k \tilde{\alpha} + \sum_{\alpha \in Cr_{i-1}(f_1)   }    k \cdot n^{f_2} (\tilde{\alpha}, \alpha   )  \alpha     ).
\end{equation}

 \begin{equation} \label{4}
 \partial_{i}^{f_2}(\tilde{\sigma })=  k \tilde{\alpha}  +   \sum_{\alpha \in Cr_{i-1}(f_1)}  k\cdot n^{f_2} (\tilde{\alpha},  \alpha) \alpha  .
\end{equation}

 \begin{equation} \label{5}
  \partial_{i-1}^{f_2}(\tilde{\alpha })= \sum_{\alpha \in Cr_{i-1}(f_2)} n^{f_2} (\tilde{\alpha}, \alpha)   \sum_{\omega \in Cr_{i-2} (f_1) }n^{f_1} (\alpha, \omega) \omega.
 \end{equation}
\end{lemma}

\begin{remark}
The above lemmas 
omit the case $q=0$,
in which  the connectedness coefficient of $n^{f_1}$ and $n^{f_2}$ may 
exchange due to the convention regarding $0$-dimensional paths (see Section~2.3). 
However, this does not alter the validity of the results.
\end{remark}

\begin{theorem}\cite[Theorem 4.1]{no.2}.
Let $f_1, f_2:K \to \R$ 
be discrete Morse functions.
Suppose that 
$h_q:C_q^{f_1} \to C_q^{f_2}$ and $g_q:C_q^{f_2} \to C_q^{f_1}$ be a birth transition and 
a death transition, 
respectively.

Then, 
$h$ and $g$ are chain maps.
\end{theorem}

Besides chain maps,
these maps 
also   induce isomorphisms in homology. 
We show that $h$ and $g$ form a 
chain-homotopy equivalence as follows.
\begin{theorem}

Suppose $h:C_\ast^{f_1} \to C_\ast^{f_2}$ and 
$g:C_\ast^{f_2} \to C_\ast^{f_1}$ are a birth transition and 
a death transition, respectively. 

Then, over $\mathbb{Q}$, 
$h$ and $g$ are chain-homotopy inverses to each other. 
That is, the induced chain maps
$h\circ g$ and  $g\circ h$
are chain homotopic to the
identities on rational coefficients.

\end{theorem}
\begin{proof}
Suppose that two additional $f_2$-critical simplices are 
$\tilde{\sigma} \in \mathrm{Cr}_i(f_2)$ 
and $\tilde{\alpha} \in \mathrm{Cr}_{i-1}(f_2)$,
with $n^{f_2}(\tilde{\sigma}, \tilde{\alpha})=k$.
For any $f_1$-critical simplex $\delta$, one checks immediately that 
$g\circ h(\delta)=\delta$. Hence 
\[ g\circ h = \mathrm{id} \quad\]
 \text{on} $C^{f_1}_\ast .$ 

On $C^{f_2}_\ast$, the only nontrivial cases are the additional pair 
$\tilde{\sigma},\tilde{\alpha}$. One computes
\[
h\circ g(x) = 
\begin{cases}
0, & x=\tilde{\sigma},\\[6pt]
-\tilde{\alpha}-\sum\limits_{\beta} n^{f_2}(\tilde{\alpha},\beta)\,\beta, & x=\tilde{\alpha},\\[10pt]
x, & \text{otherwise}.
\end{cases}
\]

Define a homotopy
\[
s: C^{f_2}_\ast \to C^{f_2}_{\ast+1}, 
\qquad
s(x) = 
\begin{cases}
-\tfrac{1}{k}\tilde{\sigma}, & x=\tilde{\alpha},\\[6pt]
0, & \text{otherwise}.
\end{cases}
\]

Here $\tfrac{1}{k}\in \mathbb{Q}$, so this definition is valid over
$\mathbb{Q}.$

Then.
\[
(\partial^{f_2}s+s\partial^{f_2})(x) = 
\begin{cases}
-\tilde{\sigma}, & x=\tilde{\sigma},\\[6pt]
-\tilde{\alpha}-\sum\limits_{\beta} n^{f_2}(\tilde{\alpha},\beta)\,\beta, & x=\tilde{\alpha},\\[10pt]
0, & \text{otherwise}.
\end{cases}
\]

Thus
\[
h\circ g - \mathrm{id} = \partial^{f_2}s+s\partial^{f_2}.
\]

Therefore $g\circ h=\mathrm{id}$ and $h\circ g\simeq \mathrm{id}$, 
as chain maps over $\mathbb{Q}$, so $h$ and $g$ are chain-homotopy inverses in this setting.

\end{proof}

See Remark 5.3 for more comments on coefficients.

\begin{lemma}
\label{lem:equivalence}
Let $f_1, f_2: K\to \mathbb{R}$
be discrete Morse functions
and $V_1, V_2$ be the corresponding gradient vector
field, respectively.
We use 
$h_q: C_q^{f_1} \to C_q^{f_2}$ 
and 
$g_q: C_q^{f_2} \to C_q^{f_1}$ 
to represent the connectedness homomorphisms.

Then, $V_1=V_2$
if and only if $h=g=\id.$
\label{f_1=f_2}
\end{lemma}

\begin{proof}
Suppose $V_1 = V_2$. 
Then, for every critical simplex $\sigma$, 
$$h_q(\sigma) = \sigma = g_q(\sigma).$$
This follows from the absence of gradient paths between 
distinct $q$-dimensional critical simplices.

Conversely, if $h = g = \mathrm{id}$, then no new critical simplices are introduced or removed, and the gradient vector fields must coincide.
\end{proof}

\begin{example}
Consider the simplicial complex $K$ with discrete Morse functions $f_1$ and $f_2$ depicted in Figure \ref{transition example}.
We follow the definition to study the connectedness homomorphisms between those two 
discrete Morse complexes.

One can check that the connectedness homomorphism 
$h: C_{\ast}^{f_1} \to C_{\ast}^{f_2}$
is a birth transition
generating 
a pair of $f_2$-critical simplices $(\tilde{\alpha}, \tilde{\sigma})$;
and 
the connectedness homomorphism 
$g: C_{\ast}^{f_2} \to C_{\ast}^{f_1}$
is a death transition
cancelling that pair of simplices.

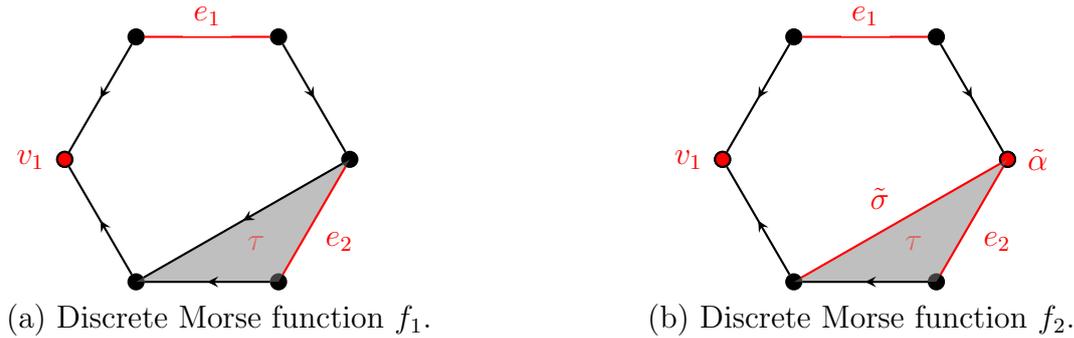
\begin{figure}[h!]

\begin{minipage}{.5\textwidth}

\begin{tikzpicture}[>=stealth, thick, scale=1.25, decoration={markings, mark=at position 0.5 with {\arrow{>}}}]
  \foreach \x in {1,2,...,6}{
    \node[draw, circle, fill=black, inner sep=2pt, label=above:] (vertex\x) at ({360/6 * (\x - 1)}:1.5cm) {};

  }

    \node[draw, circle, fill=red, inner sep=2pt, label={[text=red]left:$v_1$}] (vertex8) at ($(vertex4)$) {};

\fill[gray, opacity=0.5] (vertex1.center) -- (vertex5.center) -- (vertex6.center) -- (vertex1.center) 
node[midway, left, shift={(-0.5,-0.3)},text=red] {$\tau$};

\draw[postaction={decorate}] (vertex2) -- (vertex1);
\draw[postaction={decorate}] (vertex6) -- (vertex5);
\draw[postaction={decorate}] (vertex3) -- (vertex4);
\draw[postaction={decorate}] (vertex5) -- (vertex4);
\draw[postaction={decorate}] (vertex1) -- (vertex5);
\draw[red] (vertex2) -- (vertex3) node[midway, above, fill=white] {$e_1$} ;
\draw[red] (vertex1) -- (vertex6) node[midway, below right, fill=white] {$e_2$} ;

\end{tikzpicture}

(a) Discrete Morse function $f_1$.
\end{minipage}%
\begin{minipage}{.5\textwidth}

\centering

\begin{tikzpicture}[>=stealth, thick, scale=1.25, decoration={markings, mark=at position 0.5 with {\arrow{>}}}]
  \foreach \x in {1,2,...,6}{
    \node[draw, circle, fill=black, inner sep=2pt, label=above:] (vertex\x) at ({360/6 * (\x - 1)}:1.5cm) {};

  }

    \node[draw, circle, fill=red, inner sep=2pt, label={[text=red]left:$v_1$}] (vertex8) at ($(vertex4)$) {};
 \node[draw, circle, fill=red, inner sep=2pt, label={[text=red]right:$\tilde{\alpha}$}] (vertex9) at ($(vertex1)$) {};
\fill[gray, opacity=0.5] (vertex1.center) -- (vertex5.center) -- (vertex6.center) -- (vertex1.center) 
node[midway, left, shift={(-0.5,-0.3)},text=red] {$\tau$};

\draw[postaction={decorate}] (vertex2) -- (vertex1);
\draw[postaction={decorate}] (vertex6) -- (vertex5);
\draw[postaction={decorate}] (vertex3) -- (vertex4);
\draw[postaction={decorate}] (vertex5) -- (vertex4);
\draw[red] (vertex2) -- (vertex3) node[midway, above, fill=white] {$e_1$} ;
\draw[red] (vertex1) -- (vertex6) node[midway, below right, fill=white] {$e_2$} ;
\draw[red] (vertex1) -- (vertex5) node[midway, above left, fill=white] {$\tilde{\sigma}$} ;

\end{tikzpicture}

(b) Discrete Morse function $f_2$.
\end{minipage}

\label{transition example}
\caption{A birth-death transition.}

\label{transition example}
\end{figure}

\end{example}

Note that the formulation of birth-death transitions and Theorem 3.5
recover, on the chain level, 
the essential ideas presented
in \cite[Section 11]{Morse Theory for Cell Complexes}
regarding the cancellation 
of critical cells.
Our framework provides an explicit 
algebraic mechanism via birth-death transitions,
extending 
the intuition  of smooth Morse theory \cite{laudenbach} into the discrete combinatorial 
setting.
\section{Connecting discrete Morse functions via birth-death transitions}

In discrete Morse theory, the space of all discrete Morse functions on a simplicial complex $K$ is often studied through  their associated gradient vector fields.
Two discrete Morse functions $f_1, f_2: K \to \mathbb{R}$ are 
considered \emph{equivalent} when their gradient vector fields $V_1$ and $V_2$ coincide.

The set of all
gradient vector fields forms a simplicial complex, known as the \textit{complex of discrete Morse functions},  denoted $\M(K)$,  introduced by Chari and Joswig \cite{chari}. 
This complex parametrizes acyclic matchings (gradient vector fields) 
in the Hasse diagrams of $K$.
In particular, when 
$K$
 is a graph, this simplicial complex coincides with the \textit{complex of rooted forests},  studied by Kozlov \cite{Kozlov}.

Also in \cite{chari},  Chari-Joswig asks how to relate any two 
acyclic matchings by a sequence of finite steps.
We resolve this affirmatively on the chain level, showing that any two discrete Morse functions 
$f_1, f_2:K \to \R$
can be connected through a finite sequence of birth-death transitions, as defined in Section 3.

We first define the complex of discrete Morse functions.
\begin{definition}
Given a simplicial complex $K$,
we define the \textit{complex of discrete Morse functions} $\mathcal{M}(K)$
as the collection of subsets of edges in the Hasse diagram 
of $K$ that form acyclic matchings.

Alternatively, we may define it more explicitly:

We define a simplicial complex   $\MMK$
on $K$ as follows.
We say that a \textit{primitive vector field} of $K$
is a set  $P$ containing one  pair of simplices 
$\alpha$  and  $\sigma$ of $K$ such that 
$\alpha \prec \sigma$ and $\dim \sigma- \dim \alpha= 1$.

We let $\mathcal{P}(K)= \{P_1, P_2, \ldots, P_n  \}$ 
be the set of all primitive vector fields of $K$.

The simplices of $\MMK$ are defined as:
\begin{itemize}
\item  $0$-simplices: $\widetilde{\mathcal{M}}^0(K):= \mathcal{P}(K)$;
\item  $q$-simplices: $( P_{i_0}, P_{i_1}, \ldots, P_{i_q} )\in 
 \widetilde{\mathcal{M}}^q(K)$
if $V= \{ P_{i_0}, P_{i_1}, \ldots, P_{i_q} \}$ 
form an acyclic matching (i.e., a gradient vector field).

\end{itemize}

To ensure that all Morse functions are included, we augment
$\MMK$ with the empty simplex,
that is,
$$\mathcal{M}(K)= \MMK \cup \{\es\}.$$

\end{definition}

Note that although the empty simplex $\es$ is sometimes ignored in other papers,
this inclusion of empty simplex to $\MMK$  is 
necessary in our framework to ensure that $\M(K)$ includes all possible discrete 
Morse functions, including the discrete Morse functions where all simplices are critical.

Such discrete Morse function always exist;
for example:
the discrete Morse function
$$f: K\to \R$$
with 
$$f(\sigma)=\dim \sigma$$
corresponds to  the gradient vector field $V_0$,  in which every simplex
is critical.

Recall that in a simplicial complex,
the empty simplex $\es$
has dimension $-1$
and is a face of every vertex.
It is not hard to verify that 
if $K$ is a simplicial complex, then $K \cup \{\es\}$  is still an (abstract) simplicial complex.
We consider the poset topology of this augmented simplicial complex ordered by inclusion.

As one of the most important 
statement of Cerf theory,
any two Morse functions $\phi_1, \phi_2$ 
on the same manifold can be related by a 
finite sequence of moves involving the creation and cancellation of critical pairs.
We  can use 
the birth-death transitions
to represent a discrete version of 
this sequence between two discrete Morse functions $f_1, f_2:K \to \mathbb{R}$.
We first study the local behavior of such transitions in the complex $\mathcal{M}(K)$.

By $\overline{\sigma}_1, \overline{\sigma}_2 \in \widetilde{\mathcal{M}} (K) $,
we denote the simplex in $\MMK$
corresponding to $V_1,  V_2$,
respectively.

\begin{lemma}

\label{482}

If $\overline{\sigma}_1 \prec \overline{\sigma}_2$ 
and $\dim \overline{\sigma}_2- \dim \overline{\sigma}_1=1$,
then 
\begin{itemize}
 \item The connectedness homomorphism $h: C^{V_1}_\ast \to C^{V_2}_\ast$ is a death transition; 
 \item The connectedness  homomorphism $g: C^{V_2}_\ast \to C^{V_1}_\ast$ is a birth transition.
\end{itemize}

\end{lemma}

\begin{proof}
Suppose that 
$\overline{\sigma}_1 \prec \overline{\sigma}_2$ 
and $\dim \overline{\sigma}_2- \dim \overline{\sigma}_1=1$.
Then,
$V_1$ differs $V_2$   from a primitive vector field $(\alpha, \sigma)$.
More precisely,
$V_2=V_1\cup (\alpha, \sigma).$
This implies that
both 
$\alpha$ and $ \sigma$
are $V_1$-critical.

Note that  $(\alpha, \sigma)\in V_2$ implies
 $n^{f_1}(\sigma, \alpha)\neq 0$;
 otherwise,
one of the following contradiction arises
\begin{itemize}
\item there is no $V_1$-gradient path connecting $\sigma$ and $\alpha$; 
\item the coefficients of multiple  $V_1$-gradient  paths cancel with each other.
\end{itemize}
Case (1) contradicts with  either $\alpha\prec \sigma$;
and 
case (2)  forms a cycle with no $V_2$-critical simplices, which contradicts with the 
acyclicity 
of a gradient vector field.

Therefore, 
by definition,
the connectedness homomorphism $g$ is a birth transition,
and the connectedness homomorphism  $h$ is a death transition.
\end{proof}

\begin{example}
Figure 2 $(b)$
illustrates 
an example of the complex of discrete 
Morse functions 
obtained from the simplicial complex $K$ in $(a)$.

We let  $\overline{\sigma}_1=(a,e_1)$ and 
$\overline{\sigma}_2=((a,e_1), (b,e_2))$.
Thus, the corresponding gradient vector fields
$V_1=\{ (a,e_1) \}$
and 
$V_2=\{(a,e_1), (b,e_2) \}$.
Hence, 
one can check that
the 
connectedness homomorphisms
$h$ and $g$ 
are a death and birth transition,
respectively.

\begin{figure}[h!]
\label{complex example}

\begin{minipage}{.5\textwidth}

\begin{tikzpicture}[scale=1.5, thick]
  \node[draw, circle, fill=black, inner sep=2pt, label=left:$a$] (vertex1) at (0, 0) {};
  \node[draw, circle, fill=black, inner sep=2pt, label=right:$b$] (vertex2) at (2, 0) {};
  \node[draw, circle, fill=black, inner sep=2pt, label=above:$c$] (vertex3) at (1, 1.73) {}; 

  \draw (vertex1) -- (vertex2) node[midway, below, yshift=-5pt, fill=white] {$e_1$};
  \draw (vertex2) -- (vertex3) node[midway, right, xshift=5pt, fill=white] {$e_2$};
  \draw (vertex3) -- (vertex1) node[midway, left, xshift= -5pt, fill=white] {$e_3$};
\end{tikzpicture}

(a) Simplicial complex  $K$.
\end{minipage}%
\begin{minipage}{.5\textwidth}

\centering
\begin{tikzpicture}[scale=2, thick]
  \node[draw, circle, fill=black, inner sep=2pt, label=above left:{$(a,e_1)$}] (ae1) at (0.1,1) {};
  \node[draw, circle, fill=black, inner sep=2pt, label=above right:{$(c,e_2)$}] (ce2) at (2,1) {};
  \node[draw, circle, fill=black, inner sep=2pt, label=below right:{$(a,e_3)$}] (ae3) at (1.9,0) {};
  \node[draw, circle, fill=black, inner sep=2pt, label=below left:{$(b,e_2)$}] (be2) at (0,0) {};
  \node[draw, circle, fill=black, inner sep=2pt, label=below:{$(c,e_3)$}] (ce3) at (0.7,0.6) {};
  \node[draw, circle, fill=black, inner sep=2pt, label=above:{$(b,e_1)$}] (be1) at (2.7,0.6) {};
  
  \draw (ae1) -- (be2);
  \draw (ae1) -- (ce3);
  \draw (be2) -- (ce3);
  
   \draw (ae3) -- (be1);
  \draw (ae3) -- (ce2);
  \draw (be1) -- (ce2);
  
     \draw (ae3) -- (be2);
  \draw (ae1) -- (ce2);
  \draw (be1) -- (ce3);
\end{tikzpicture}

(b)  $\MMK$
\end{minipage}

\caption{An example of discrete Morse function complex.}

\end{figure} 

\end{example}

With regard to the empty simplex $\es$,
we state the following  lemma.

\begin{lemma}
\label{0-dim}
Let $v$ be any vertex of $\M(K)$,
and $V$ be the  corresponding   gradient vector field.

Then, 
\begin{enumerate}
\item the connectedness homomorphism $h_0: C_{\ast} ^{V} \to C_{\ast} ^{V_0}$
is a birth transition;
\item  the connectedness homomorphism $g_0: C_{\ast} ^{V_0} \to C_{\ast} ^{V}$
is a death transition,
\end{enumerate}
where $V_0$ 
is the empty gradient vector field.

\end{lemma}

\begin{proof}
The proof is straightforward from definitions.
\end{proof}

We illustrate this in Figure 3 and 4,
where the birth-death transitions
are drawn in dashed lines.

\begin{example}
Let $K$ be an edge.
$\MMK$ consists of two points, thus is not connected.
Hence,
there is no path of simplices of $\MMK$
connecting those two vertices.

By inserting $\es$
to $\MMK$, as we do in Fig.4,
the path
$(a,e_1) \leftrightarrow \es \leftrightarrow (b,e_1)$
connects two vertices.
And the path represents 
a composition of a birth transition 
and a death transition.

\begin{figure}[h]

\begin{minipage}{.5\textwidth}

\begin{tikzpicture}[scale=1.5, thick]
  \node[draw, circle, fill=black, inner sep=2pt, label=left:$a$] (vertex1) at (0, 0) {};
  \node[draw, circle, fill=black, inner sep=2pt, label=right:$b$] (vertex2) at (2, 0) {};

  \draw (vertex1) -- (vertex2) node[midway, below, yshift=-5pt, fill=white] {$e_1$};
 
\end{tikzpicture}

(a) Simplicial complex  $K$.
\end{minipage}%
\begin{minipage}{.5\textwidth}

\centering
\begin{tikzpicture}[scale=1.5, thick]
  \node[draw, circle, fill=black, inner sep=2pt, label=left:{$(a,e_1)$}]   (vertex1) at (0, 0) {};
  \node[draw, circle, fill=black, inner sep=2pt, label=right:{$(b,e_1)$}] (vertex2) at (2, 0) {};

\end{tikzpicture}

(b)   $\MMK$
\end{minipage}

\caption{An example of discrete Morse function complex.}
\label{complex example}
\end{figure}
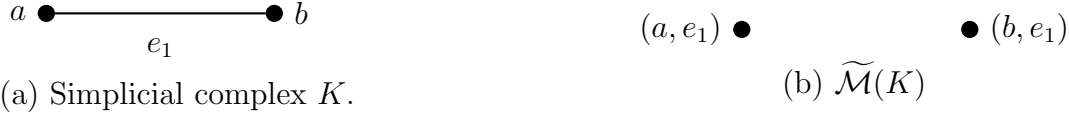

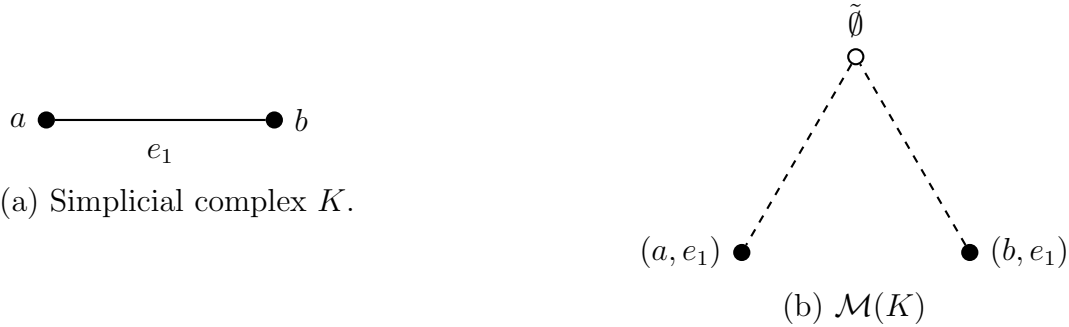
\begin{figure}[h!]

\begin{minipage}{.5\textwidth}

\begin{tikzpicture}[scale=1.5, thick]
  \node[draw, circle, fill=black, inner sep=2pt, label=left:$a$] (vertex1) at (0, 0) {};
  \node[draw, circle, fill=black, inner sep=2pt, label=right:$b$] (vertex2) at (2, 0) {};

  \draw (vertex1) -- (vertex2) node[midway, below, yshift=-5pt, fill=white] {$e_1$};
 
\end{tikzpicture}

(a) Simplicial complex  $K$.
\end{minipage}%
\begin{minipage}{.5\textwidth}

\centering
\begin{tikzpicture}[scale=1.5, thick]
  \node[draw, circle, fill=black, inner sep=2pt, label=left:{$(a, e_1)$}] (vertex1) at (0, 0) {};
  \node[draw, circle, fill=black, inner sep=2pt, label=right:{$(b,e_1)$}] (vertex2) at (2, 0) {};
  \node[draw, circle, inner sep=2pt, label=above:$\tilde{\emptyset}$] (vertex3) at (1, 1.73) {}; 

  \draw[dashed] (vertex2) -- (vertex3) node[midway, right, xshift=5pt, fill=white] {};
  \draw[dashed] (vertex3) -- (vertex1) node[midway, left, xshift= -5pt, fill=white] {};
\end{tikzpicture}

(b)  $\M(K)$
\end{minipage}

\caption{The empty simplex.}
\label{complex example}
\end{figure}

\end{example}

Combining 
 the above lemmas,
we recover the idea of Cerf theory on $K$
 by  connecting the discrete Morse functions with birth-death transitions.
\begin{theorem}
\label{main_theorem}
Let $K$ be a finite simplicial complex
and 
$f_1, f_2: K \to \mathbb{R}$ be
discrete Morse functions.

Then, $f_1$ and $f_2$ can be transformed into 
one another
	with a finite sequence of 
birth-death transitions.
\end{theorem}

\begin{proof}
Let $V_1$ and $V_2$ be 
the gradient vector fields 
corresponding to $f_1$ and $f_2$. 
Suppose that these correspond to simplices 
$\overline{\sigma}_1, \overline{\sigma}_2 \in \mathcal{M}(K)$.
 
Then, we can construct a  path consisting of successive adjacent simplices
$$ \overline{\sigma}_1^{(n)} \succ  \alpha_1^{(n-1)} \succ \alpha_2^{(n-2)}\succ \cdots \succ \alpha_{n}^{0} \succ \es$$ 
 connecting  
$ \overline{\sigma}_1$ to $\es$,
denoted $\overline{\sigma}_1 \;\to\; \es$.
Each adjacent pair  
corresponds to   a birth transition
as established in Lemma \ref{482}(2), when $\dim \alpha_i \geq 1$;
and  established in Lemma \ref{0-dim}(1), when $\dim \alpha_i = 0$.

Similarly, 
we can also construct a  path 
connecting  
$\es$ to $ \overline{\sigma}_2$,
 each step corresponding to a death transition.

Hence, both $\overline{\sigma}_1$ and $\overline{\sigma}_2$ 
can be connected to $\es$ via such  paths:
\[
\overline{\sigma}_1 \;\to\; \es 
\quad \text{and} \quad 
\es \;\to\; \overline{\sigma}_2 .
\]
Concatenating these two paths gives a  sequence
\[
\overline{\sigma}_1 \;\to\; \es \;\to\; \overline{\sigma}_2 ,
\]
lying entirely in $\M(K)$, in which each adjacent pair of simplices 
corresponds to a birth or death transition. 
Since $K$ is finite, this sequence has finitely many steps, 
completing the proof.

\end{proof}

\section{An alternative proof of Forman's Theorems}
In Forman's original theory,  a simplicial complex  $K$ is simplified through the collapse of gradient 
pairs (pairs of non-critical simplices)
guided by the descending values of the discrete Morse function $f$,
or  equivalently,  the arrows  in the associated gradient vector field $V$.
This collapsing sequence
reduces $K$ to a smaller simplicial complex 
with critical simplices encoding its homology.

Dually,
we can equivalently generate/cancel the critical simplices, 
the generators of the discrete Morse complex $C_{\ast}^f$,
in pairs similar to cancellation in classical Morse theory.
In this process, $K$ remains unchanged,
and the homological structure is modified through algebraic transitions.

Theorem \ref{main_theorem} establishes 
that any two gradient vector fields in
$\M(K) $ are path-connected via birth-death transitions.
Besides the algebraic framework given in algebraic Morse theory, 
we aim to provide more geometric information and 
intuitions by the introduction of connectedness homomorphisms
and birth-death transitions.

One of the 
fundamental results in discrete Morse theory, 
established by Forman \cite{Forman_guide}, is the equivalence
 between discrete Morse homology and simplicial homology. 
 In our framework, where birth-death transitions act as chain maps between discrete Morse complexes, 
 this correspondence can be recovered over coefficient 
$ \mathbb{Q}$
 naturally.

\begin{theorem} \cite[Theorem 7.1]{Forman_guide}
\label{thm:homology-isomorphism}
Let $K$ be a simplicial complex  and 
$f: K \to \R$
be a discrete Morse function.

Then,  over $\mathbb{Q}$,
$$H_{\ast}(K) \cong H_{\ast}^{\text{simp}}(K),$$
where 
$H_{\ast}^{\text{simp}}(K)$
denotes the simplicial homology of $K$.
\end{theorem}

\begin{proof}
Let $V_0$ be the empty gradient vector field 
(i.e., the one where all simplices are critical).
Define an 
isomorphism 
$$\psi_q: C_q ^{\text{simp}}  \to C_{q} ^{V_0},$$
for all $q$,
where $C_q ^{\text{simp}} $ denotes the
$q$-th simplicial chain group of $K$.
Note that $$\psi_q= \id_q,$$
for all $q$,
as the generators of $C_q ^{\text{simp}} $ and 
$C_{q} ^{V_0} $ coincide.

By Theorem~\ref{main_theorem}, the gradient vector field $V$ corresponding to  $f$
 is connected to $V_0$ 
 through a sequence of birth-death transitions, 
 each of which defines a chain homotopy equivalence over $\mathbb{Q}$. 
 Thus, composing the induced maps yields a homology isomorphism:
 \[H_{\ast}^V(K) \cong H_{\ast}^{V_0}(K) \cong H_{\ast}^{\text{simp}}(K),\]
 completing the proof.
\end{proof}

Besides Theorem 5.1,
other Forman's theorems can also be recovered in the sense
of ``collapsing critical simplices" while keeping $K$ unchanged, similar to the proof above.
We provide one more example, the Morse equality, as follows.

\begin{theorem}\cite[Theorem 2.11]{Forman_guide}
\label{thm:morse_equality}
Let $K$ be a finite simplicial complex and 
$f:K \to \R$ be a discrete Morse function.

Then, over $\mathbb{Q}$,
$$\sum_q (-1)^q \beta_q(K)=\sum_q (-1)^q Cr_q(f),$$
where 
$\beta_q(K)$ denotes the $q$-th Betti number 
of $K$.
\end{theorem}

\begin{proof}
Let $f_0: K\to \R$
be a discrete Morse function whose every simplex is critical.
 Theorem \ref{main_theorem} implies that 
there exists a finite sequence of birth-death transitions
transforming $f_0$ to $f$.
By definition,
a birth/death transition generates/cancels 
a pair of $(q,q+1)$-dimension critical simplices. 
Thus,
the alternating sum remains invariant under such transitions.
That is,
$$\sum_q (-1)^q Cr_q(f)= \sum_q (-1)^q Cr_q(f_0).$$

Moreover,
since 
$$ \sum_q (-1)^q Cr_q(f_0)=\chi(K)=\sum_q (-1)^q \beta_q,$$
where $\chi(K)$ denotes the Euler characteristic of $K$,
we have the desired equality.

\end{proof}
\begin{remark}
We make the following comments regarding coefficients:

Forman’s Theorems~\ref{thm:homology-isomorphism} and~\ref{thm:morse_equality} 
are known to hold integrally, i.e.\ over $\mathbb{Z}$. 
Theorem~3.6 establishes that a birth–death transition induces a chain homotopy equivalence over $\mathbb{Q}$. 
The proof works integrally whenever the connectedness coefficient
 satisfies $|k|=1$, so in this case 
the statement holds over $\mathbb{Z}$, as well.

In smooth Cerf theory, bifurcations with $|k|\geq 2$ can be decomposed into successive bifurcations of type $|k|=1$. 
We believe an analogous statement  hold in the discrete setting as well. 
However, it may require further subdivisions of the simplicial complex $K$. 
A detailed investigation of this question will be carried out in future work.

\end{remark}

\subsection*{Several remarks on topology of the set of discrete Morse functions.}

Recall that in \cite{no.2},
we introduced the following equivalence relation 
on the set of all discrete Morse functions on a simplicial complex
$K$,
denoted, $\mathcal{F}$,
to study its topological properties.

\begin{definition}
Let $f_1, f_2, \ldots,  f_n : K \to \mathbb{R}$
be a sequence of 
discrete Morse functions,
with  corresponding 
discrete Morse complexes $C_{\ast}^{f_1},C_{\ast}^{f_2} ,\ldots, C_{\ast}^{f_n}$.
If, for each $i= 1, 2, \cdots, n-1$, 
the connectedness homomorphism
$h^i_{\ast}: C_{\ast}^{f_i} \to C_{\ast}^{f_{i+1}}$
is either an isomorphism or a  
birth/death
 transition,
then we call the 
sequence 
of discrete Morse functions 
a \textit{birth-death transition sequence}.
Any two 
discrete Morse functions $f_j$ and $f_k$
in such a sequence
are called 
\textit{transitively connected},
denoted $f_j \Tsim  f_k$.
\end{definition}

This defines a natural equivalence relation:
\begin{proposition}
Transitive connectedness $\Tsim$ is an equivalence relation on $\mathcal{F}$.
\end{proposition}

\begin{proof}
Reflexivity holds since $C_{\ast}^{f} \cong C_{\ast}^{f}$ via the identity map.
 Symmetry follows from the reversibility of birth-death transitions and the invertibility of isomorphisms.
Transitivity is built into the definition of the sequence.

\end{proof}

This equivalence relation allows us to
 characterize the topology of 
$\mathcal{F}$ via the quotient
$\mathcal{F}/ \Tsim$.

\begin{theorem}

Let $K$ be a simplicial complex
and $\mathcal{F}$ be the set of all discrete Morse functions on $K$.

Then,
$$\mathcal{F}/ \Tsim= \{\ast \}.$$

\end{theorem}

\begin{proof}
By Theorem \ref{main_theorem}, any two discrete Morse functions
$f_1,f_2\in \mathcal{F}$
can be connected by a finite sequence of birth-death transitions.
Each transition
$h_{\ast}^{i}:C_{\ast}^{f_i} \to C_{\ast}^{f_{i+1}}$
is chain map,
and isomorphisms arise when gradient vector fields are identical (Lemma \ref{lem:equivalence}).
Thus, 
$f_1\Tsim f_2$
for all 
$f_1,f_2\in \mathcal{F}$,
implying the quotient consists of a single point.
That is,
$$\mathcal{F}/ \Tsim= \{\ast \}.$$

\end{proof}

\begin{corollary}
Let $K$ be a simplicial complex.
Then,
the augmented complex of discrete Morse functions $\M(K)$ 
is connected with poset topology.

\end{corollary}


\begin{remark}
Note that higher-dimensional homology of $\M(K)$ and $\MMK$
are isomorphic,
as $\es$ affects only connectivity at dimension $0$.
This viewpoint aligns with the results of Scoville and Zaremsky \cite[Theorem 5.4]{higher connectivity of the morse complex}.
They
proved that almost all $\MMK$ (except Example 4.5) are connected,
which implies Corollary 5.6.

\end{remark}


\newpage
\begin{thebibliography}{9}


















\bibitem{determined}
Capitelli, N.A., Minian, E.G. A Simplicial Complex is Uniquely Determined by Its Set of Discrete Morse Functions. Discrete Comput Geom 58, 144–157 (2017). https://doi.org/10.1007/s00454-017-9865-z

\bibitem{cerf2}
J. Cerf,
Sur les difféomorphismes de la sphère de dimension trois $(\Gamma_4=0)$,
Lecture Notes in Mathematics, vol. 53, Berlin-New York: Springer-Verlag.

\bibitem{cerf}
J. Cerf,
 La stratification naturelle des espaces de fonctions différentiables réelles et le théorème de la pseudo-isotopie. Publications Mathématiques de l'IHÉS, Tome 39 (1970), pp. 5-173. 

\bibitem{chari}
M. K. Chari,  M.  Joswig, 
Complexes of discrete Morse functions,
Discrete Mathematics 302 (2005) 39-51.

\bibitem{homotopy_type}
C. Donovan,  M. Lin,  N. A. Scoville,
On the homotopy and strong homotopy type of complexes of discrete Morse functions. 
Canadian Mathematical Bulletin. 2023;66(1):19-37. doi:10.4153/S0008439522000121.

\bibitem{homotopy_type 2}
C. Donovan, N. A. Scoville,
Star clusters in the matching, Morse, and generalized complex of discrete Morse functions.
New York J. Math. 29 (2023) 1393–1412.


\bibitem{Forman_guide} 
R. Forman,
A User's Guide to Discrete Morse Theory.
Seminaire Lotharingien de Combinatoire 48 (2002), Article B48c.

\bibitem{Morse Theory for Cell Complexes}
R. Forman,
Morse Theory for Cell Complexes.
Advances in Mathematics 134, (1998), 90-145.


\bibitem{Guest}
M. Guest, 
Morse theory in the 1990's,
arXiv:math/0104155.



\bibitem{birth and death} 
H. King, 
K. Knudson, 
N.Mramor Kosta, 
Birth and death in discrete Morse theory, J. Symb. Comput. 78 (2017) 41–60.

\bibitem{Kozlov}
D. N.  Kozlov, 
Complexes of directed trees, 
J. Comb. Theory Ser. A 88 (1) (1999) 112–122.



\bibitem{Discrete Morse Theory Kozlov}
D. N.  Kozlov,
Organized Collapse: An Introduction to 
Discrete Morse Theory,
Graduate Studies in Mathematics, 207.
American Mathematical Society, Providence, RI, 2020.


\bibitem{laudenbach}
F.  Laudenbach,
Appendix. On the Thom-Smale complex,
Astérisque, tome 205 (1992),  219-233.


\bibitem{Milnor}
J. Milnor, Morse theory.  Annals of Mathematics
Studies, No. 51, Princeton University Press, Princeton, N.J., 1963.








\bibitem{higher connectivity of the morse complex}
N. A. Scoville, M. C. B. Zaremsky,
Higher connectivity of the morse complex,
Proceeding of AMS, Series B,
Volume 9, Pages 135–149.



\bibitem{no.1}
C. Zheng,
Persistent pairs and strong connectedness in discrete Morse functions on simplicial complex \rom{1},  Topology and its Applications.



\bibitem{no.2}
C. Zheng,
The Connectedness Homomorphisms between Discrete Morse Complexes,
Topology and its Applications, 2024, Volume 355, Article 109022.




\end{thebibliography}
\end{document}